\documentclass[12pt]{article}
\usepackage[top=2.5cm,bottom=2.5cm,left=2.9cm,right=2.9cm]{geometry}
\usepackage{amssymb}
\usepackage{amsmath,amsthm}
\usepackage{graphicx}
\usepackage{enumerate}
\usepackage{tikz}
\usepackage{mathrsfs}
\usepackage{verbatim}

\usepackage{colortbl}

\newtheorem{remark}{Remark}[section]

\newtheorem{lemma}[remark]{Lemma}
\newtheorem{theorem}[remark]{Theorem}

\newtheorem{corollary}[remark]{Corollary}

\usepackage{eucal}
\newcommand{\n}{\operatorname{n}}
\newcommand{\m}{\operatorname{m}}
\newcommand{\ecc}{\operatorname{ecc}}
\newcommand{\epn}{\operatorname{epn}}

\title{On the $2$-packing differential of a graph}

\author{A. Cabrera Mart\'inez$^{(1)}$, M.L. Puertas$^{(2)}$, J. A. Rodr\'{\i}guez-Vel\'{a}zquez$^{(1)}$\\
\\
$^{(1)}${\small Universitat Rovira i Virgili,}
{\small Departament d'Enginyeria Inform\`atica i Matem\`atiques } \\  {\small Av. Pa\"{\i}sos
Catalans 26, 43007 Tarragona, Spain.} \\{\small
  abel.cabrera\@@urv.cat, juanalberto.rodriguez\@@urv.cat}\\
$^{(2)}${\small Universidad de Almer\'ia, Departamento de Matem\'aticas}\\
 {\small
Carretera Sacramento s/n, 04120, Almer\'ia, Spain.} \\ {\small
 mpuertas\@@ual.es}
}

\date{ }
\begin{document}
\maketitle

\begin{abstract}
Let $G$ be a graph of order $\n(G)$ and vertex set $V(G)$. Given a set  $S\subseteq V(G)$, we define the external neighbourhood  of $S$  as the set $N_e(S)$ of all vertices in $V(G)\setminus S$ having at least one neighbour in $S$.
 The differential of $S$ is defined to be $\partial(S)=|N_e(S)|-|S|$. In this paper, we introduce the study of the $2$-packing differential of a graph, which we define  as
$\partial_{2p}(G)=\max\{\partial(S): S\subseteq V(G) \text{ is a }2\text{-packing}\}.$
We show that the $2$-packing differential is closely related to several graph parameters, including the packing number, the independent domination number, the total domination number, the perfect differential, and the unique response Roman domination number. In particular, we show that the theory of $2$-packing differentials is an appropriate framework to  investigate the unique response  Roman domination number of a graph without the use of functions.
 Among other results, we obtain a Gallai-type theorem, which states  that $\partial_{2p}(G)+\mu_{_R}(G)=\n(G)$, where  $\mu_{_R}(G)$ denotes the unique response Roman domination number of $G$. As a consequence of the study, we derive several combinatorial results on $\mu_{_R}(G)$, and we show that the problem of finding this parameter is NP-hard.  In addition, the particular case of lexicographic product graphs is discussed.
\end{abstract}


\section{Introduction}

Given a graph $G=(V(G),E(G))$, the \emph{open neighbourhood} of a vertex $v$ of a graph $G$ is defined to be $N(v)=\{u \in V(G):\; u \text{ is adjacent to } v\}$.
The \emph{open neighbourhood of a set}  $S\subseteq V(G)$ is defined as
$N(S)=\cup_{v\in S}N(v)$, while the \emph{external neighbourhood} of $S$, or boundary of $S$,  is defined as $N_e(S)=N(S)\setminus S$.
The \emph{closed neighbourhood} of $v$ is $N[v]=N(v)\cup\{v\}$ and the \emph{closed neighbourhood} of a set $S\subseteq V(G)$ is $N[S]=N(S)\cup S$. We denote by $\deg(v)=|N(v)|$ the degree of vertex $v$, as well as $\delta(G)=\min_{v \in V(G)}\{\deg(v)\}$ the minimum degree of $G$, $\Delta(G)=\max_{v \in V(G)}\{\deg(v)\}$ the maximum degree of $G$,  $\n(G)=|V(G)|$ the order of $G$ and $\m(G)=|E(G)|$ the size of $G$.

The \emph{differential of a set} $S\subseteq V(G)$ is defined as $\partial(S)=|N_e(S)|-|S|$, while the \emph{differential of a graph} $G$ is defined to be
$$\partial(G)=\max \{\partial(S):\, S\subseteq V(G)\}.$$
As described in \cite{MR2212796}, the definition of $\partial(G)$ was given by Hedetniemi about
twenty-five years ago in an unpublished paper, and was also considered by Goddard
and Henning \cite{MR1605086}.
 After that, the differential of a graph has been studied by several authors, including \cite{Sergio-differential-2014,PerfectDifferential,MR2212796,MR2724896,MR3444564}

Lewis et al.\ \cite{MR2212796} motivated the definition of differential from the
following game, what we call \emph{graph differential game}. ``\emph{You are allowed to buy as many tokens as you like, say $k$ tokens, at a cost of one dollar each. You then place the tokens on some
subset $D$ of $k$ vertices of a graph $G$. For each vertex of $G$ which has no token on
it, but is adjacent to a vertex with a token on it, you receive
one dollar. Your
objective is to maximize your profit, that is, the total value received
minus the cost of the tokens bought}". Obviously, $\partial(D)=|N_e(D)|-|D|$ is the profit obtained with the placement  $D$, while the maximum profit equals $\partial(G)$.

In order to introduce a variant of this game, we need the following terminology. A  set $S\subseteq V(G)$   is a $2$-\emph{packing} if $N[u]\cap N[v]= \varnothing$ for every pair of different vertices $u,v\in S$. We define
 $$\wp(G)=\{S\subseteq V(G): S \text{ is a }2\text{-packing  of } G \}.$$
The $2$-\emph{packing number} of $G$, denoted by  $\rho(G)$, is defined to be
$$\rho(G)=\max \{|S|: S\in \wp(G)\}.$$

Now, we consider a version of the graph differential game in which we impose two additional conditions: (1) every vertex which has no token on it is adjacent to at most one vertex with a token on it; and (2) no vertex with a token on it is adjacent to a vertex with a token on it. In this case, any placement $D$ of tokens is a $2$-packing and so this version of the game can be called $2$-\emph{packing graph differential game}, as the maximum profit equals the $2$-\emph{packing differential} of $G$, which is  defined  as
$$\partial_{2p}(G)=\max \{\partial(S): S\in \wp(G)\}.$$

We define a  $\partial_{2p}(G)$-set as a set $S\in \wp(G)$
with $\partial(S)=\partial_{2p}(G)$. The same agreement will be assumed for optimal parameters associated to other characteristic sets defined in the paper.

Let $G$ be the graph shown in Figure \ref{Figure-2-packingGame}.  In the graph differential game, if we place the tokens on vertices $a$, $b$ and $c$, then
we obtain the maximum profit $\partial(G)=7$. In contrast, in the
$2$-packing graph differential game, the maximum profit $\partial_{2p}(G)=6$ is given by the placement of token on vertices  $a$ and $b$.

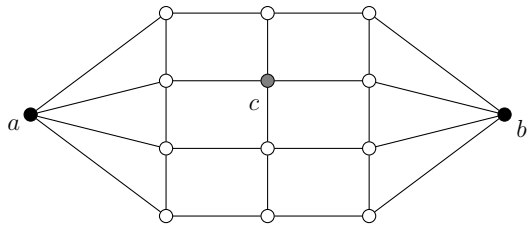
\begin{figure}[ht]
\centering
\begin{tikzpicture}[scale=.45, transform shape]
\node [draw, shape=circle, fill=black] (aa1) at  (-4,3) {};
\node [below] at (-4.5,3) {\LARGE $a$};
\node [draw, shape=circle, fill=black] (cc1) at  (10,3) {};
\node [below] at (10.5,3) {\LARGE $b$};

\node [draw, shape=circle] (a1) at  (0,0) {};
\node [draw, shape=circle] (a2) at  (0,2) {};
\node [draw, shape=circle] (a3) at  (0,4) {};
\node [draw, shape=circle] (a4) at  (0,6) {};

\node [draw, shape=circle] (b1) at  (3,0) {};
\node [draw, shape=circle] (b2) at  (3,2) {};
\node [draw, shape=circle, fill=gray] (b3) at  (3,4) {};
\node [below] at (2.6,3.6) {\LARGE $c$};
\node [draw, shape=circle] (b4) at  (3,6) {};

\node [draw, shape=circle] (c1) at  (6,0) {};
\node [draw, shape=circle] (c2) at  (6,2) {};
\node [draw, shape=circle] (c3) at  (6,4) {};
\node [draw, shape=circle] (c4) at  (6,6) {};

\draw (aa1)--(a1)--(b1)--(c1)--(cc1)--(c2)--(b2)--(a2)--(aa1);
\draw (aa1)--(a3)--(b3)--(c3)--(cc1)--(c4)--(b4)--(a4)--(aa1);
\draw(a1)--(a2)--(a3)--(a4);
\draw(b1)--(b2)--(b3)--(b4);
\draw(c1)--(c2)--(c3)--(c4);
\end{tikzpicture}
\caption{A graph $G$ with $\partial(G)=7$ and $\partial_{2p}(G)=6$.}\label{Figure-2-packingGame}
\end{figure}

In this paper we show that the $2$-packing differential is closely related to several graph parameters, including the packing number, the independent domination number, the total domination number, the perfect differential, and the unique response Roman domination number. In particular, we show that the theory of $2$-packing differentials is an appropriate framework to  investigate the unique response  Roman domination number of a graph without the use of functions.

The rest of the paper is organized as follows. Section~\ref{SectionGeneral Bonds} is devoted to provide some general bounds on the $2$-packing differential, in terms of different parameters such as the maximum and minimum degrees of the graph, the number of vertices, the number of edges, among others. We show that the bounds are tight and, in some cases, we characterize the graphs achieving them. As a consequence of the study, we show that the problem of finding $\partial_{2p}(G)$ is NP-hard. In Section~\ref{SectionNordhaus-Gaddum} we obtain a Nordhaus-Gaddum type theorem for both the addition and the product of the $2$-packing differential of a graph and its complement. Section~\ref{SectionPerfectDiff} is devoted to the study of the relationship between the $2$-packing differential and the perfect differential of a graph. In Section~\ref{SectionGallai} we prove a Gallai-type theorem which states  that $\partial_{2p}(G)+\mu_{_R}(G)=\n(G)$, where  $\mu_{_R}(G)$ denotes the unique response Roman domination number of $G$.
We derive several consequences of this result, including the fact that the problem of finding $\mu_{_R}(G)$ is NP-hard. We finally show the case of the lexicographic product graph in Section~\ref{Sectionlexicographic}, where we obtain general lower and upper bounds of the $2$-packing differential and the unique response Roman domination number. We also compute the exact value of both parameters for the lexicographic product of a path and any other graph.

\section{Basic results and computational complexity}\label{SectionGeneral Bonds}

In this section we present lower and upper bounds of the $2$-packing differential of a graph. In some cases, we also characterize the graphs achieving such bounds. Moreover, these results will provide the computational complexity of the computation of the $2$-packing differential.

We begin with the following theorem that will be a key result in the rest of the paper. The bounds presented here are in terms of the maximum and minimum degrees of the graph and the $2$-packing number.

\begin{theorem}\label{th-packings}
If $G$ is a graph  with no isolated vertex, then the following statements hold.
\begin{enumerate}[{\rm (i)}]
\item $\partial_{2p}(G)=\displaystyle\max_{S\in \wp(G)}\left\{\sum_{v\in S}(deg(v)-1)\right\}.$
\item If $G$ is a $k$-regular graph, then
$\partial_{2p}(G)=(k-1)\rho(G).$
\item $ \rho(G)(\delta(G)-1)\leq \partial_{2p}(G)\leq \rho(G)(\Delta(G)-1).$

\item $\partial_{2p}(G)=\rho(G)(\Delta(G)-1)$ if and only if there exists a $\rho(G)$-set $S$ such that $\deg(v)=\Delta(G)$ for every $v\in S$.

\item If $\partial_{2p}(G)=\rho(G)(\delta(G)-1)$, then every vertex $v$ in every $\rho(G)$-set has degree $\deg(v)=\delta(G)$.
\end{enumerate}
\end{theorem}

\begin{proof}
Notice that for any  $S\in \wp(G)$, we have that  $ \displaystyle \partial(S)=|N(S)|-|S|=\sum_{v\in S}(deg(v)-1)$. Hence,
$
\partial_{2p}(G) =\displaystyle \max_{S\in \wp(G)}\left\{\partial (S)\right\}
 =\max_{S\in \wp(G)}\left\{|N(S)|-|S|\right\}
 =\max_{S\in \wp(G)} \{\sum_{v\in S}(deg(v)-1)\}.
$
Therefore, (i) follows. From (i) we derive (ii) and (iii).

We proceed to prove (iv). If there exists a $\rho(G)$-set $S$ such that $\deg(v)=\Delta(G)$ for every $v\in S$, then  $\rho(G)(\Delta(G)-1)=\partial(S)\le \partial_{2p}(G)$. Hence, (iii) leads to $\partial_{2p}(G)=\rho(G)(\Delta(G)-1)$. Conversely, if
$\partial_{2p}(G)=\rho(G)(\Delta(G)-1)$, then for any $\partial_{2p}(G)$-set $D$ we have that $\rho(G)(\Delta(G)-1)=\partial(D)=\displaystyle\sum_{v\in D}(deg(v)-1)\le \displaystyle\sum_{v\in D}(\Delta(G)-1)\le |D|(\Delta(G)-1)\le \rho(G)(\Delta(G)-1)$, which implies that $D$ is a $\rho(G)$-set and $\deg(v)=\Delta(G)$ for every $v\in D$.

Finally, we prove (v). Suppose on the contrary that there exists a $\rho(G)$-set $S'$ and a vertex $v\in S'$ such that $\deg(v)> \delta(G)$. Then, $\partial_{2p}(G)=\displaystyle\max_{S\in \wp(G)}\left\{\sum_{v\in S}(deg(v)-1)\right\}\geq \sum_{v\in S'}(deg(v)-1)>\rho (G)(\delta(G)-1)$. Figure \ref{Figura-CotaInfPacking-GradMinim} shows an example of a graph $G$ with $\partial_{2p}(G)= \rho(G)(\delta(G)-1)$, the set of black-coloured vertices is the unique $\rho(G)$-set.
\end{proof}

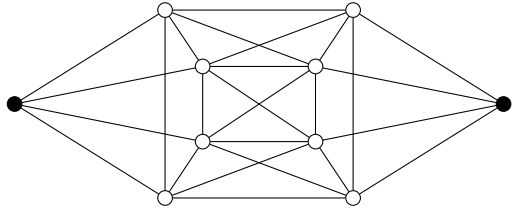
\begin{figure}[ht]
\centering
\begin{tikzpicture}[scale=.5, transform shape]
\node [draw, shape=circle, fill=black] (aa1) at  (-4,3) {};
\node [draw, shape=circle, fill=black] (bb1) at  (9,3) {};

\node [draw, shape=circle] (a1) at  (0,0.5) {};
\node [draw, shape=circle] (a2) at  (1,2) {};
\node [draw, shape=circle] (a3) at  (1,4) {};
\node [draw, shape=circle] (a4) at  (0,5.5) {};

\node [draw, shape=circle] (b1) at  (5,0.5) {};
\node [draw, shape=circle] (b2) at  (4,2) {};
\node [draw, shape=circle] (b3) at  (4,4) {};
\node [draw, shape=circle] (b4) at  (5,5.5) {};

\draw (aa1)--(a1)--(b1)--(bb1)--(b2)--(a2)--(aa1);
\draw (aa1)--(a3)--(b3)--(bb1)--(b4)--(a4)--(aa1);
\draw(a1)--(a2)--(a3)--(a4)--(a1);
\draw(b1)--(b2)--(b3)--(b4)--(b1);
\draw(a1)--(b2)--(a3)--(b4);
\draw(b1)--(a2)--(b3)--(a4);
\end{tikzpicture}
\caption{A graph $G$ with $\partial_{2p}(G)= \rho(G)(\delta(G)-1)$.}\label{Figura-CotaInfPacking-GradMinim}
\end{figure}

The first consequence of the preceding result is the determination of the complexity of the computation of the $2$-packing differential. Given a graph $G$ and a positive integer $t$, the \textit{$2$-packing problem} is to decide
whether there exists a $2$-packing $S$ in $G$ such that $|S|$ is at least
$t$. It is well known that the $2$-packing  problem is NP-complete. Hence, the optimization problem of finding $\rho(G)$ is NP-hard. Furthermore, it was shown in \cite{MR1267280} that the $2$-packing remains NP-complete for regular bipartite graphs. Therefore, from
Theorem \ref{th-packings} (ii) we derive the following result.

\begin{corollary}\label{2-packing differential}
The problem of finding the $2$-packing differential of a graph is NP-hard, even for regular bipartite  graphs.
\end{corollary}

As a second consequence of Theorem~\ref{th-packings} we can derive the exact value of the $2$-packing differential in particular graph families. To this end, we need the following definitions.

A set $S\subseteq V(G)$ of vertices is a \emph{dominating set}  if
$N(v)\cap S\ne \varnothing$ for every vertex   $v\in V(G)\setminus S$.
Let $\mathcal{D}(G)$ be the set of dominating sets of $G$.
The \emph{domination number} of $G$ is defined to be,
$$\gamma(G)=\min\{|S|:\, S\in \mathcal{D}(G)\}.$$

Observe that  $\gamma(G)\ge \rho(G)$ and  $\mathcal{D}(G) \cap \wp(G)\ne \varnothing$
if and only if there exists a $\gamma(G)$-set which is a $\rho(G)$-set. A graph $G$ with $\mathcal{D}(G) \cap \wp(G)\ne \varnothing$ is called an  \emph{efficient closed domination graph} and in such case $\gamma(G)= \rho(G)$ . Furthermore,  Meir and Moon \cite{MR0401519} showed that $\gamma(T)= \rho(T)$ for every tree $T$.  Even so, there are trees in which $\mathcal{D}(T) \cap \wp(T)= \varnothing$.

If $G$ is a $k$-regular nonempty graph and there exists  $S\in \mathcal{D}(G)\cap \wp(G)$, then $\n(G)=|S|+|N(S)|=|S|+|S|k=\rho(G)(1+k).$
Therefore, the following result is a direct consequence of Theorem \ref{th-packings} (ii).
\begin{corollary}
Let $k$ be a positive integer. If $G$ is a $k$-regular efficient closed domination graph, then
$$ \partial_{2p}(G)=\frac{(k-1)\n(G)}{k+1}.$$
\end{corollary}

The next statement is a direct consequence of Theorem \ref{th-packings} (i).

\begin{corollary}
For any graph $G$ of diameter at most two, $\partial_{2p}(G)=\Delta(G)-1$.
\end{corollary}

We can also involve the order $n(G)$ and the size $m(G)$ of the graph and we obtain the following upper bound of the $2$-packing differential.

\begin{theorem} \label{Bound size and order}
For any graph $G$ with no isolated vertex,  $$ \partial_{2p}(G)\le \rho(G)(\delta(G)-1)+2\m(G)-\n(G)\delta(G).$$
Furthermore, the equality holds if and only if there exists a $\rho(G)$-set $S$ such that $\deg(v)=\delta(G)$ for every $v\in V(G)\setminus S$.
\end{theorem}

\begin{proof}
If $S\subseteq V(G)$ is a $\partial_{2p}(G)$-set, then
\begin{align*}
2\m(G)&=\sum_{v\in S}\deg(v)+\sum_{v\not \in S}\deg(v)\\
&=\partial_{2p}(G)+|S|+\sum_{v\not \in S}\deg(v)\\
&\ge \partial_{2p}(G)+|S|+(\n(G)-|S|)\delta(G)\\
&=\partial_{2p}(G)+\n(G)\delta(G)+|S|(1-\delta(G))\\
&\ge \partial_{2p}(G)+\n(G)\delta(G)+\rho(G)(1-\delta(G)).
\end{align*}
Therefore, the upper bound follows.

From the previous  procedure  we deduce that if $\partial_{2p}(G)= 2\m(G)-\n(G)\delta(G)+(\delta(G)-1)\rho(G)$, then $|S|=\rho(G)$ and  $\deg(v)=\delta(G)$ for every $v\in V(G)\setminus S$.
Conversely, if $S$ is a $\rho(G)$-set and $\deg(v)=\delta(G)$ for every $v\in V(G)\setminus S$, then the same procedure leads to $2\m(G)-\n(G)\delta(G)+(\delta(G)-1)\rho(G)=\partial(S)\le \partial_{2p}(G)$, and by the upper bound we conclude that the equality holds.
\end{proof}
Notice that  for any $k$-regular graph, $2\m(G)=\n(G)\delta(G)=n(G)k$, which implies that
$\partial_{2p}(G)=(k-1)\rho(G)=2\m(G)-\n(G)\delta(G)+(\delta(G)-1)\rho(G).$


\begin{figure}[ht]
\centering
\begin{tikzpicture}[scale=.5, transform shape]
\node [draw, shape=circle, fill=black] (aa1) at  (-4,3) {};
\node [draw, shape=circle, fill=black] (bb1) at  (7,3) {};
\node [draw, shape=circle, fill=black] (dd1) at  (18,3) {};

\node [draw, shape=circle] (a1) at  (0,0) {};
\node [draw, shape=circle] (a2) at  (0,2) {};
\node [draw, shape=circle] (a3) at  (0,4) {};
\node [draw, shape=circle] (a4) at  (0,6) {};

\node [draw, shape=circle] (b1) at  (3,0) {};
\node [draw, shape=circle] (b2) at  (3,2) {};
\node [draw, shape=circle] (b3) at  (3,4) {};
\node [draw, shape=circle] (b4) at  (3,6) {};

\draw (aa1)--(a1)--(b1)--(bb1)--(b2)--(a2)--(aa1);
\draw (aa1)--(a3)--(b3)--(bb1)--(b4)--(a4)--(aa1);
\draw(a1)--(a2);
\draw(a3)--(a4);
\draw(b1)--(b2);
\draw(b3)--(b4);

\node [draw, shape=circle] (c1) at  (11,0) {};
\node [draw, shape=circle] (c2) at  (11,2) {};
\node [draw, shape=circle] (c3) at  (11,4) {};
\node [draw, shape=circle] (c4) at  (11,6) {};

\node [draw, shape=circle] (d1) at  (14,0) {};
\node [draw, shape=circle] (d2) at  (14,2) {};
\node [draw, shape=circle] (d3) at  (14,4) {};
\node [draw, shape=circle] (d4) at  (14,6) {};

\draw (bb1)--(c1)--(d1)--(dd1)--(d2)--(c2)--(bb1);
\draw (bb1)--(c3)--(d3)--(dd1)--(d4)--(c4)--(bb1);
\draw(c1)--(c2);
\draw(c3)--(c4);
\draw(d1)--(d2);
\draw(d3)--(d4);
\end{tikzpicture}
\caption{A graph $G$ with $ \partial_{2p}(G)=\rho(G)(\delta(G)-1)+2\m(G)-\n(G)\delta(G).$ The set of black-coloured vertices forms a $\partial_{2p}(G)$-set.}\label{FigureUpperBound-1-Size}
\end{figure}
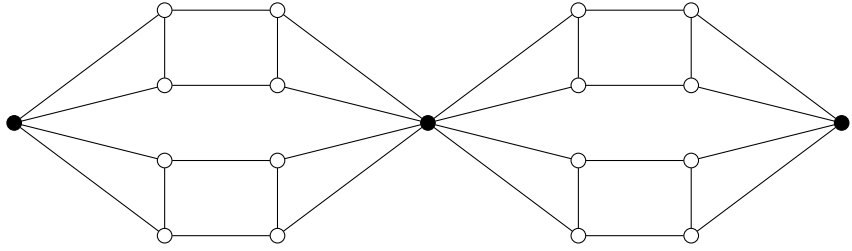

Figure \ref{FigureUpperBound-1-Size} shows a graph $G$ with $ \partial_{2p}(G)=\rho(G)(\delta(G)-1)+2\m(G)-\n(G)\delta(G).$
Observe that the set $S$ of black-coloured vertices forms a $\rho(G)$-set with $\deg(v)=\delta(G)$ for every $v\in V(G)\setminus S$.

The bound presented in Theorem~\ref{Bound size and order} has a special behaviour in efficient closed domination graphs, as we show in the following two results.

\begin{theorem}\label{Lower bound packing=dominating}

If $G$ is an efficient closed domination graph with no isolated vertex, then the following statements hold.
\begin{enumerate}[{\rm (i)}]
\item
$\partial_{2p}(G)\ge \n(G)-2\rho(G).$

\item If there exists a $\rho(G)$-set  $S$ such that $\deg(v)=\delta(G)$ for every $v\in V(G)\setminus S$, then
$$\partial_{2p}(G)= \n(G)-2\rho(G).$$
\end{enumerate}
\end{theorem}

\begin{proof}
For any set  $D\in \mathcal{D}(G)\cap \wp(G)$,
$$\partial_{2p}(G)\ge  \partial(D)= |N(D)|-|D|= |V(G)\setminus D|-|D|=\n(G)-2\rho(G).$$
Now, in order to prove (ii), let $S$ be a  $\rho(G)$-set such that $\deg(v)=\delta(G)$ for every $v\in V(G)\setminus S$. In the proof of Theorem~\ref{Bound size and order}  we have shown that $S$ is a $\partial_{2p}(G)$-set. Notice that $|N(S)|\le \n(G)-|S|=\n(G)-\rho(G)$.  Hence,   $\n(G)-2\rho(G)\le \partial_{2p}(G)=\partial(S)=|N(S)|-\rho(G)\le \n(G)-2\rho(G)$. Therefore, $\partial_{2p}(G)= \n(G)-2\rho(G).$
\end{proof}

By Theorems \ref{Bound size and order} and \ref{Lower bound packing=dominating} we deduce the following result.

\begin{theorem}\label{Lower bound packing} For any efficient closed domination graph $G$ with $\delta(G)\ge 1$,
$$\rho(G)\ge \left\lceil \frac{\n(G)(\delta(G)+1)-2\m(G)}{\delta(G)+1}\right\rceil.$$
Furthermore, the equality
holds if and only if there exists a $\rho(G)$-set  $S$ such that $\deg(v)=\delta(G)$ for every $v\in V(G)\setminus S$.
\end{theorem}

\begin{proof}
By combining the bounds given in Theorems  \ref{Bound size and order} and \ref{Lower bound packing=dominating}, we deduce that
$2\m(G)-\n(G)\delta(G)+(\delta(G)-1)\rho(G) \ge \partial_{2p}(G)\ge \n(G)-2\rho(G).$
Therefore, the bound follows.

Now, if $\rho(G)= \frac{\n(G)(\delta(G)+1)-2\m(G)}{\delta(G)+1},$ then we have equalities in the above inequality chain, and  by Theorem  \ref{Bound size and order} we conclude that there exists a $\rho(G)$-set  $S$  such that $\deg(v)=\delta(G)$ for every $v\in V(G)\setminus S$.

Conversely, if there exists a $\rho(G)$-set  $S$  such that $\deg(v)=\delta(G)$ for every $v\in V(G)\setminus S$, then Theorems  \ref{Bound size and order} and  \ref{Lower bound packing=dominating} lead to $2\m(G)-\n(G) \delta(G)+(\delta(G)-1)\rho(G)= \n(G)-2\rho(G),$ which implies that $\rho(G)= \frac{\n(G)(\delta(G)+1)-2\m(G)}{\delta(G)+1}$.
\end{proof}

We now present lower and upper bounds of the $2$-packing differential in terms of some dominating-type parameters. We begin with the following definition.

A set $S\subseteq V(G)$ which is both dominating and independent is called an \emph{independent dominating set}. Moreover, the \emph{independent domination number} $i(G)$ is
$$i(G)=\min\{|S|:\, S \text{ is an independent dominating set }\}.$$

In the following result we present both bounds and we characterize graphs attaining the upper one.

\begin{theorem} \label{ThDominationIndep2}
For any graph $G$ with no isolated vertex, $$\Delta(G)-1\le \partial_{2p}(G) \leq \frac{(\Delta(G)-1)(\n(G)-i(G))}{\Delta(G)}.$$
Furthermore, the following statements are equivalent.
\begin{enumerate}
\item[{\rm (i)}] $ \partial_{2p}(G)=\frac{(\Delta(G)-1)(\n(G)-i(G))}{\Delta(G)}.$
\item[{\rm (ii)}] There exist an $i(G)$-set $D$ and a $2$-packing $S\subseteq D$ such that $N(S)\cup D=V(G)$ and $deg(v)=\Delta(G)$ for every $v\in S$.
\end{enumerate}
\end{theorem}

\begin{proof}
The lower bound is obtained by definition of $\partial_{2p}(G)$.  Let $S$ be a $\partial_{2p}(G)$-set, $S'=\{v\in V(G)\setminus N[S]: N(v)\subseteq N(S)\}$ and $S''=V(G)\setminus (N[S]\cup S')$. Let $S_1$ be an independent dominating set of the subgraph induced by $S''$. Notice that this subgraph does not have isolated vertices, which implies  that $|S_1|<|S''|$. Also, observe that  the set $S\cup S'\cup S_1$ is an independent dominating set of $G$. Hence,
$$|N(S)|=\n(G)-(|S|-|S'|-|S''|)\le \n(G)-(|S|-|S'|-|S_1|)\le \n(G)-i(G).$$
Also, notice that $|N(S)|\leq \Delta(G)|S|$. Hence,
\begin{align*}
\Delta(G)\partial_{2p}(G)&= \Delta(G)(|N(S)|-|S|)\\
                & = \Delta(G)|N(S)|-\Delta(G)|S|\\
                &\le \Delta(G)|N(S)|-|N(S)|\\
    &= (\Delta(G)-1)|N(S)|\\
    &\le (\Delta(G)-1)(\n(G)-i(G)).
\end{align*}
Hence,  the upper bound follows.

Now, if (i) holds, then we have equalities in the previous inequality chains. From these, we deduce that $S''=\varnothing$, $D=S\cup S'$ is an $i(G)$-set, where $S$ is a $2$-packing and $|N(S)|=\Delta(G)|S|=n-i(G)$, which implies that $N(S)\cup D=V(G)$ and $deg(v)=\Delta(G)$ for every $v\in S$.  Therefore, (ii) follows.

Finally, if (ii) holds, i.e., there exist an $i(G)$-set $D$ and a $2$-packing $W\subseteq D$ such that $N(W)\cup D=V(G)$ and $deg(v)=\Delta(G)$ for every $v\in W$, then  $\Delta(G)|W|=|N(W)|=\n(G)-i(G)$ and so
$$\partial_{2p}(G)\geq \partial(W)=|N(W)|-|W|=\n(G)-i(G)-\frac{\n(G)-i(G)}{\Delta(G)}.$$
Therefore, from the upper bound we deduce that the equality holds.
\end{proof}


\begin{corollary}
Let $G$ be a graph with no isolated vertex. If $\Delta(G)=\n(G)-i(G)$, then $\partial_{2p}(G)=\Delta(G)-1$.
\end{corollary}

The converse of the result above does not hold. For instance,  let $G$ be a graph having a vertex $v\in V(G)$ of maximum degree  $deg(v)=\Delta(G)=\n(G)-3$ such that the two vertices in   $V(G)\setminus N[v]$ are adjacent. In this case, $i(G)=2$ and $\partial_{2p}(G)=\n(G)-4=\Delta(G)-1$, while $\Delta(G)<\n(G)-i(G)$.

Next, we consider a particular case of the equivalence given in Theorem \ref{ThDominationIndep2}.

\begin{theorem} \label{ThDominationIndep2-particular-case}
For any graph  $G$  with no isolated vertex, $$ \partial_{2p}(G)\leq \n(G)-i(G)-1.$$
Furthermore, the equality holds if and only if $\Delta(G)=\n(G)-i(G)$.
\end{theorem}

\begin{proof}
From  Theorem  \ref{ThDominationIndep2} we deduce that $\n(G)-i(G)\geq \Delta(G)$ and $\partial_{2p}(G)\leq  \n(G)-i(G)-\frac{\n(G)-i(G)}{\Delta(G)}.$ Therefore, the upper bound follows.

Now, if $\Delta(G)=\n(G)-i(G)$, then from the   bounds given in Theorem \ref{ThDominationIndep2} we deduce that $\partial_{2p}(G)=\n(G)-i(G)-1$.  Conversely, if  $\partial_{2p}(G)=\n(G)-i(G)-1$, then  by Theorem  \ref{ThDominationIndep2} we deduce that $\Delta(G)=\n(G)-i(G)$, which completes the proof.
\end{proof}

We now characterize all trees attaining the lower bound presented  in Theorem~\ref{ThDominationIndep2}. To this end, we will use the following notation. Given a tree $T$, a vertex $v\in V(T)$ with $\deg(v)=\Delta(T)$ and any vertex $v_i\in N(v)$, denote by $T_i(v)$ the sub-tree of $T$ rooted in $v_i$ obtained from $T$ by removing the edge $\{v,v_i\}$. For each $i\in \{1, \dots ,\Delta (T)\}$ such that $T_i(v)$ is non-trivial, denote by $w_i$ the neighbour of $v_i$ in $T_i(v)$ with the largest degree.

\begin{theorem}
Let $T$ be a non-trivial tree. Then, $\partial_{2p}(T)=\Delta(T)-1$ if and only if, for each vertex $v\in V(T)$ with $\deg(v)=\Delta(T)$, the following conditions hold.

\begin{enumerate}
\item[{\rm (i)}] $\ecc(v)\leq 3$,
\item[{\rm (ii)}] $\deg(v_i) +\sum_{j\neq i} (\deg (w_j) -1)\leq \Delta(T)$, ($i=1, \dots, \Delta (T)$ and $j\neq i$ with $T_j(v)$ non-trivial),
\item[{\rm (iii)}] $\sum_{i} (\deg (w_i) -1)\leq \Delta(T)-1$, ($i\in \{1, \dots ,\Delta (T)\}$ such that $T_i(v)$ is non-trivial).
\end{enumerate}
\end{theorem}

\begin{proof}
Assume that $\partial_{2p}(T)=\Delta(T)-1$ and take $v\in V(T)$ with $\deg(v)=\Delta(T)$. The distance between two vertices $x,y\in V(T)$ will be denoted by $d(x,y)$.

If there exists $w\in V(T)$ such that $d(v,w)=4$, then for $u\in N(w)$ with $d(v,u)=3$ the set $\{v,u\}$ is a $2$-packing and  $\deg(u)\geq 2$. Thus,   $\partial_{2p}(T)\ge \partial(\{v,u\})\ge (\deg(u)-1)+(\deg(v)-1)>\deg(v)-1=\Delta(T)-1$, which is a contradiction. Therefore, $\ecc(v)\leq 3$.

The vertex set $X=\{v_i\}\cup \{w_j\colon j\neq i, T_j(v) \text{ is non-trivial}\}$ is a $2$-packing, so $\partial(X)=(\deg(v_i)-1) +\sum_{j\neq i} (\deg (w_j) -1)\leq \partial_{2p}(T)=\Delta(T)-1$, that gives (ii).
Analogously, the set $\{w_i\colon T_i(v) \text{ is non-trivial}\}$ is a $2$-packing and this gives (iii).

Now, let $T$ be a non-trivial tree satisfying (i), (ii) and (iii), let $S$ be a $2$-packing of $T$ and let $v$ be a vertex with $\deg(v)=\Delta (T)$. Since $\partial_{2p}(T)\ge \Delta(T)-1$, we proceed to show that  $\partial_{2p}(T)\le \Delta(T)-1$. To this end, we differentiate the following three cases.

Case 1: $v\in S$. If there exists $u\in S\setminus\{v\}$, then $d(u,v)\geq 3$ and condition (i) gives that $\deg(u)=1$ (see Figure~\ref{FigureTree}). Therefore, $\partial(S)=\sum_{w\in S}(deg(w)-1)=\deg(v)-1=\Delta (T)-1$.

Case 2: $v_i\in S$, for some $v_i\in N(v)$. Notice that $N[v]\cap S=\{v_i\}$ and, if there exists $u\in S\setminus\{v_i\}$, then $d(u,v_i)\geq 3$. In such a case, $u$ belongs to a non-trivial rooted sub-tree $T_j(v)$, $j\neq i$ and $d(u,v)\geq 2$ (see Figure~\ref{FigureTree}). Clearly, $S\cap V(T_j(v))=\{u\}$ and from (ii) we obtain $\partial(S)=\sum_{w\in S}(deg(w)-1)\leq (\deg(v_i)-1) +\sum_{j\neq i} (\deg (w_j) -1)\leq \Delta(T)-1$.

Case 3: $N[v]\cap S=\emptyset$. Then every vertex  $u\in S$ satisfies $d(u,v)\geq 2$ and from (iii) we obtain $\partial(S)=\sum_{w\in S}(deg(w)-1)\leq \sum_{i} (\deg (w_i) -1)\leq \Delta(T)-1$.
\end{proof}

\begin{figure}[ht]
\centering
\begin{tikzpicture}[scale=.4, transform shape,rotate=-90]
\node [draw, shape=circle, fill=black] (a) at  (0,6) {};

\node [draw, shape=circle, fill={rgb:black,1;white,2}] (a1) at  (2,2) {};
\node [draw, shape=circle] (a2) at  (2,4) {};
\node [draw, shape=circle] (a3) at  (2,6) {};
\node [draw, shape=circle] (a4) at  (2,8) {};
\node [draw, shape=circle] (a5) at  (2,11) {};

\node [draw, shape=circle,] (b11) at  (4,1) {};
\node [draw, shape=circle] (b12) at  (4,3) {};

\node [draw, shape=circle, fill={rgb:black,1;white,2}] (b21) at  (4,4) {};

\node [draw, shape=circle] (b31) at  (4,6) {};

\node [draw, shape=circle] (b41) at  (4,7) {};
\node [draw, shape=circle, fill={rgb:black,1;white,2}] (b42) at  (4,9) {};

\node [draw, shape=circle, fill={rgb:black,1;white,2}] (b51) at  (4,10) {};
\node [draw, shape=circle] (b52) at  (4,12) {};

\node [draw, shape=circle,fill=black] (c111) at  (6,1) {};

\node [draw, shape=circle,fill=black] (c421) at  (6,9) {};

\node [draw, shape=circle,fill=black] (c511) at  (6,10) {};

\draw(a)--(a1);
\draw(a)--(a2);
\draw(a)--(a3);
\draw(a)--(a4);
\draw(a)--(a5);

\draw(a1)--(b11);
\draw(a1)--(b12);

\draw(a2)--(b21);

\draw(a3)--(b31);

\draw(a4)--(b41);
\draw(a4)--(b42);

\draw(a5)--(b51);
\draw(a5)--(b52);

\draw(b11)--(c111);

\draw(b42)--(c421);

\draw(b51)--(c511);

\end{tikzpicture}
\caption{A tree with $\partial_{2p}(T)=\Delta (T)-1 =4$. The set of black-coloured vertices is a $\partial_{2p}(T)$-set and the set of grey-coloured vertices is another one.}\label{FigureTree}
\end{figure}
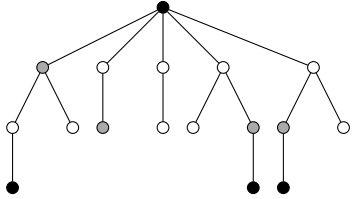

Some upper bounds of the $2$-packing differential can be also expressed in terms of the total domination number. A set $S\subseteq V(G)$ is a \emph{total dominating set} of a graph $G$ without isolated vertices  if every vertex $v\in V(G)$ is adjacent to at least one vertex in $S$. Let $\mathcal{D}_t(G)$ be the set of total dominating sets of $G$.
The \emph{total domination number} of $G$ is defined to be,
$$\gamma_t(G)=\min\{|S|:\, S\in \mathcal{D}_t(G)\}.$$
By definition, $\mathcal{D}_t(G)\subseteq \mathcal{D}(G)$, so that $\gamma(G)\le \gamma_t(G)$. Furthermore, $\gamma_t(G)\le 2\gamma(G)$.

\begin{lemma}\label{LemaDuro}
Let $G$ be a graph with no isolated vertex.
Let $S$ be a $\partial_{2p}(G)$-set such that $|S|$ is maximum among all $\partial_{2p}(G)$-sets, $S'=\{v\in V(G)\setminus N[S]: N(v)\subseteq N(S)\}$ and $S''=V(G)\setminus (N[S]\cup S')$.
The following statements hold.
\begin{enumerate}[{\rm (i)}]
\item $N(v)\cap N(S)\neq \varnothing$ for every $v\in S''$.
\item $\gamma_t(G)\leq  2|S|+|S'|+|S''|$.
\item If $\gamma_t(G)= 2|S|+|S'|+|S''|$, then $S'\cup S''=\varnothing.$
\end{enumerate}
\end{lemma}

\begin{proof}
Suppose that there exists a vertex $v\in S''$ such that $N(v)\subseteq S''$. Notice that $S_v=S\cup \{v\}$ is a $2$-packing of $G$ and that $\partial(S_v)\geq \partial(S)=\partial_{2p}(G)$, which is a contradiction because $|S_v|>|S|$. Hence, for every $v\in S''$, it follows that $N(v)\cap N(S)\neq \varnothing$.

Now, we define $W'$ as a set of minimum cardinality among the sets $W\subseteq N(S)$ satisfying  that $N(x)\cap  W\ne \varnothing$ for every $x\in S\cup S'\cup S''$.

Notice that $S \cup W'$ is a total dominating set of $G$, which implies that $\gamma_t(G)\leq |S| + |W'|\leq 2|S|+|S'|+|S''|$.

In order to prove (iii) assume $\gamma_t(G)= 2|S|+|S'|+|S''|$. In this case, $|W'|=|S|+|S'|+|S''|$. Suppose that there exists $x\in S'\cup S''$.
Since there exists  $y\in N(x)\cap W'\cap N(S)$,  the minimality of $|W'|$ leads to  $ |W'|< |S|+|S'|+ |S''|$, which is a contradiction. Therefore, $S'\cup S''=\varnothing.$
\end{proof}

\begin{theorem} \label{BounGammaTotal}
For any graph $G$ with no isolated vertex,  $$ \partial_{2p}(G)\leq \n(G)-\gamma_t(G).$$
Furthermore, the equality holds  if and only if $G$ is an efficient closed domination graph with $\gamma_t(G)=2\gamma(G)$.
\end{theorem}

\begin{proof}
Let $S$ be a $\partial_{2p}(G)$-set such that $|S|$ is maximum among all $\partial_{2p}(G)$-sets, $S'=\{v\in V(G)\setminus N[S]: N(v)\subseteq N(S)\}$ and $S''=V(G)\setminus (N[S]\cup S')$. By Lemma \ref{LemaDuro},  $\gamma_t(G)\leq 2|S|+|S'|+|S''|$. Hence,
$$\partial_{2p}(G)= |N(S)|-|S|=\n(G)-2|S|-|S'|-|S''|\le \n(G)-\gamma_t(G).$$
Therefore, the upper bound follows.

Now, if $\mathcal{D}(G)\cap \wp(G)\ne \varnothing $ and  $\gamma_t(G)=2\gamma(G)$, then from the upper bound above and the lower bound given in Theorem \ref{Lower bound packing=dominating} we deduce that $\partial_{2p}(G)=\n(G)-\gamma_t(G)$.  Conversely, if  $\partial_{2p}(G)=\n(G)-\gamma_t(G)$, then  we have equalities in the above inequality chain. In particular, we deduce that $2|S|+|S'|+|S''|=\gamma_t(G)$, and Lemma \ref{LemaDuro} leads to   $S'\cup S''=\varnothing$. Hence, $S\in \mathcal{D}(G)\cap \wp(G)$ and so,  $\gamma_t(G)=2|S|=2\gamma(G)$, which completes the proof.
\end{proof}

\begin{theorem}\label{Second bound size and order}
For any graph $G$ with no isolated vertex,
$$ \partial_{2p}(G)\le \frac{2\m(G)+\rho(G)(\delta(G)-1)-\gamma_t(G)\delta(G)}{\delta(G)+1}.$$
Furthermore, the equality holds if and only if  $\gamma_t(G)=2\gamma(G)$ and  there exists  a set  $S\in \mathcal{D}(G)\cap \wp(G)$ such that $\deg(v)=\delta(G)$ for every $v\in V(G)\setminus S$.
\end{theorem}

\begin{proof}
Let $S$ be a $\partial_{2p}(G)$-set such that $|S|$ is maximum among all $\partial_{2p}(G)$-sets, and let $m'$ be the size of the subgraph of $G$ induced by $V(G)\setminus S$.
 Notice that   every vertex in $N(S)$ has exactly one neighbour in $S$, which implies that  $$m'=\frac{1}{2}\left(\displaystyle \sum_{v\in N(S)}(\deg(v)-1)+\sum_{v\in V(G)\setminus N[S]}\deg(v)\right).$$
 Observe that Lemma \ref{LemaDuro} leads to  $\gamma_t(G)\le 2|S|+|V(G)\setminus N[S]|$. Hence,
\begin{align*}
\m(G)&=\sum_{v\in S}\deg(v)+m'\\
&=\left(\partial_{2p}(G)+|S|\right)+\frac{1}{2}\left(\displaystyle \sum_{v\in N(S)}(\deg(v)-1)+\sum_{v\in V(G)\setminus N[S]}\deg(v)\right),\, \text{ by Theorem \ref{th-packings} (i),}
\\
&\ge \partial_{2p}(G)+|S| +\frac{1}{2}\left(|N(S)|(\delta(G)-1)+(\gamma_t(G)-2|S|)\delta(G)\right)
\\
&\ge \partial_{2p}(G)+\frac{1}{2}(|S|+(|S|-|N(S)|)+(|N(S)|-|S|)\delta(G)+\gamma_t(G)\delta(G)-|S|\delta(G))\\
&= \partial_{2p}(G)+\frac{1}{2}\left( \partial_{2p}(G)(\delta(G)-1)-(\delta(G)-1)|S| + \gamma_t(G)\delta(G)  \right)\\
&\ge \partial_{2p}(G)+\frac{1}{2}\left( \partial_{2p}(G)(\delta(G)-1)-(\delta(G)-1)\rho(G) + \gamma_t(G)\delta(G)  \right).
\end{align*}
Therefore, the upper bound follows.

Now, if the bound is achieved, then from the previous procedure  we deduce that $\deg(v)=\delta(G)$ for every $v\in V(G)\setminus S$, $|S|=\rho(G)$ and $\gamma_t(G)=2|S|+|V(G)\setminus N[S]|=\n(G)-\partial_{2p}(G)$. Thus, by  Theorem \ref{BounGammaTotal} we conclude that
$\mathcal{D}(G)\cap \wp(G)\ne \varnothing $ and  $\gamma_t(G)=2\gamma(G)$.
Conversely, if $\gamma_t(G)=2\gamma(G)$ and there exists a set $D\in \mathcal{D}(G)\cap \wp(G)$ such that $\deg(v)=\delta(G)$ for every $v\in V(G)\setminus D$, then the same procedure leads to $$2\m(G)+\rho(G)(\delta(G)-1)-\gamma_t(G)\delta(G)=(\delta(G)+1)\partial (D)\le (\delta(G)+1)\partial_{2p}(G),$$ and by the upper bound we conclude that the equality holds.
\end{proof}

To conclude this section, we derive some bounds in terms of the order of $G$.

\begin{theorem} \label{TrivialBound}
The following statements hold for any graph $G$ with no isolated vertex.
\begin{enumerate}[{(i)}]
\item If $k\in \{1,2,3\}$ and $\Delta(G)=\n(G)-k$, then $\partial_{2p}(G)= \n(G)-k-1.$

\item If $\Delta(G)\le \n(G)-4$, then $\partial_{2p}(G)\le \n(G)-4$ and the equality holds if and only if $G$ is an efficient closed domination graph with  $\gamma(G)=2$.
\end{enumerate}
\end{theorem}

\begin{proof}
 Assume  $\Delta(G)=\n(G)-k$. If $k\in\{1,2\}$, then $\rho(G)=1$ and so  $\partial_{2p}(G)=\Delta(G)-1=\n(G)-k-1$.
Now, let $S$ be a $\partial_{2p}(G)$-set.
If $k=3$ then either $\rho(G)=1$, and we proceed as above, or $|S|\ge 2$, which implies that $\n(G)-4=\Delta(G)-1\le
\partial_{2p}(G)=\partial (S)= |N(S)|-|S|\le \n(G)-2|S|\le \n(G)-4.$
 Therefore, (i) follows.

From now on, assume  $\Delta(G)\le \n(G)-4$. Let  $S$ be a $\partial_{2p}(G)$-set.
If $|S|=1$, then  $\partial_{2p}(G)= \Delta(G)-1\le \n(G)-5.$ Now, if $|S|\ge 2$, then $\partial_{2p}(G)=\partial (S)= |N(S)|-|S|\le \n(G)-2|S|\le \n(G)-4.$ Therefore, the equality $\partial_{2p}(G)=\partial(S)=\n(G)-4$ holds if and only if $S\in \mathcal{D}(G)\cap \wp(G)$ and $|S|=2$.
\end{proof}

\section{Nordhaus-Gaddum type relations} \label{SectionNordhaus-Gaddum}

Nordhaus and Gaddum \cite{Nordhaus-Gaddum} in 1956 proposed lower and upper bounds, in terms of the order of the graph, on the sum and the product of the chromatic number of a graph and its complement. Since then, several inequalities of a similar type have been proposed for other graph parameters. In this section we derive some Nordhaus-Gaddum type relations for the $2$-packing differential.

A set $S\subseteq V(G)$ is an \emph{open packing}  if $N(u)\cap N(v)=\varnothing$ for every pair of different  vertices $u,v\in S$.
 The \emph{open packing number} of $G$, denoted by $\rho_o(G)$, is the maximum cardinality among all open packings of $G$.
A graph $G$ is an \emph{efficient open domination graph} if there exists
an open packing   which is a dominating set.

Notice that the complement of $G$, denoted by $G^c$, is an efficient closed domination graph with  $\gamma(G^c)=2$ if and only if $G$ is an efficient open domination graph with $\rho_o(G)=2$.

By Theorem \ref{TrivialBound} and the lower bound given in Theorem \ref{ThDominationIndep2}, we deduce the following result.

\begin{theorem}
Given a graph  $G$, the following statements hold.
\begin{itemize}
\item  If  $\Delta(G)= \n(G)-2$ and $\delta(G)=1$, then
$ \partial_{2p}(G)+\partial_{2p}(G^c)=2\n(G)-6$ and $ \partial_{2p}(G)\cdot \partial_{2p}(G^c)=(\n(G)-3)^2$.

\item  If  $\Delta(G)= \n(G)-2$ and $\delta(G)=2$, then
$ \partial_{2p}(G)+\partial_{2p}(G^c)=2\n(G)-7$ and $\partial_{2p}(G)\cdot \partial_{2p}(G^c)=(\n(G)-3)(\n(G)-4)$.

\item  If  $\Delta(G)= \n(G)-3$ and $\delta(G)=1$, then
$ \partial_{2p}(G)+\partial_{2p}(G^c)=2\n(G)-7$ and $ \partial_{2p}(G)\cdot \partial_{2p}(G^c)=(\n(G)-3)(\n(G)-4)$.

\item  If  $\Delta(G)= \n(G)-3$ and $\delta(G)=2$, then
$ \partial_{2p}(G)+\partial_{2p}(G^c)=2\n(G)-8 $ and $ \partial_{2p}(G)\cdot \partial_{2p}(G^c)=(\n(G)-4)^2$.

\item  If  $\Delta(G)= \n(G)-2$ and $\delta(G)\ge 3$, then
$$2\n(G)-\delta(G)-5\le \partial_{2p}(G)+\partial_{2p}(G^c)\leq 2\n(G)-7,$$
$$(\n(G)-3)(\n(G)-\delta (G)-2)\le \partial_{2p}(G)\cdot \partial_{2p}(G^c)\leq (\n(G)-3)(\n(G)-4).$$

\item  If  $\Delta(G)= \n(G)-2$ and $\delta(G)\ge 3$, then
$ \partial_{2p}(G)+\partial_{2p}(G^c)= 2\n(G)-7$ if and only if  $\partial_{2p}(G)\cdot \partial_{2p}(G^c)= (\n(G)-3)(\n(G)-4)$
if and only if  $G$ is an efficient open domination graph with $\rho_o(G)=2$.

\item   If  $\Delta(G)= \n(G)-3$ and $\delta(G)\ge 3$, then
$$2\n(G)-\delta(G)-6\le \partial_{2p}(G)+\partial_{2p}(G^c)\leq 2\n(G)-8,$$
$$(\n(G)-4)(\n(G)-\delta (G)-2)\le \partial_{2p}(G)\cdot \partial_{2p}(G^c)\leq (\n(G)-4)^2.$$

\item  If  $\Delta(G)= \n(G)-3$ and $\delta(G)\ge 3$, then
$\partial_{2p}(G)+\partial_{2p}(G^c)= 2\n(G)-8 $ if and only if $\partial_{2p}(G)\cdot \partial_{2p}(G^c)= (\n(G)-4)^2$ if and only if  $G$ is an efficient open domination graph with $\rho_o(G)=2$.

\item  If  $\Delta(G)\le  \n(G)-4$ and $\delta(G)=1$, then
$$\n(G)+\Delta(G)-4\le \partial_{2p}(G)+\partial_{2p}(G^c)\leq 2\n(G)-7,$$
$$(\Delta(G)-1)(\n(G)-3)\le \partial_{2p}(G)\cdot \partial_{2p}(G^c)\leq (\n(G)-3)(\n(G)-4).$$

\item  If  $\Delta(G)\le  \n(G)-4$ and $\delta(G)=1$, then
$ \partial_{2p}(G)+\partial_{2p}(G^c)= 2\n(G)-7 $ if and only if $\partial_{2p}(G)\cdot \partial_{2p}(G^c)=(\n(G)-4)(\n(G)-3)$
if and only if $G$ is an efficient closed domination graph with  $\gamma(G)=2$.

\item  If  $\Delta(G)\le  \n(G)-4$ and $\delta(G)=2$, then
$$\n(G)+\Delta(G)-5\le \partial_{2p}(G)+\partial_{2p}(G^c)\leq 2\n(G)-8,$$
$$(\Delta(G)-1)(\n(G)-4)\le \partial_{2p}(G)\cdot \partial_{2p}(G^c)\leq (\n(G)-4)^2.$$

\item  If  $\Delta(G)\le \n(G)-4$ and $\delta(G)= 2$, then
$\partial_{2p}(G)+\partial_{2p}(G^c)= 2\n(G)-8 $ if and only if $\partial_{2p}(G)\cdot \partial_{2p}(G^c)=(\n(G)-4)^2$ if and only if $G$ is an efficient closed domination graph with  $\gamma(G)=2$.

\item  If  $\Delta(G)\le \n(G)-4$ and $\delta(G)\ge 3$, then
$$\n(G)+\Delta(G)-\delta(G)-3\le \partial_{2p}(G)+\partial_{2p}(G^c)\leq 2\n(G)-8.$$
$$(\Delta(G)-1)(\n(G)-\delta (G)-2)\le \partial_{2p}(G)\cdot \partial_{2p}(G^c)\leq (\n(G)-4)^2.$$

\item  If  $\Delta(G)\le \n(G)-4$ and $\delta(G)\ge 3$, then
$\partial_{2p}(G)+\partial_{2p}(G^c)= 2\n(G)-8 $ if and only if $\partial_{2p}(G)\cdot \partial_{2p}(G^c)= (\n(G)-4)^2$ if and only if $G$ is simultaneously an efficient closed domination and  an efficient open domination graph with  $\rho_o(G)=\gamma(G)=2$.

\end{itemize}

\end{theorem}

\section{Perfect differential versus $2$-packing differential }\label{SectionPerfectDiff}

In this section we show some relationships between the $2$-packing differential and the perfect differential of a graph. Given a set $S\subseteq V(G)$ and a vertex $v\in S$, the \emph{external private neighbourhood} $\epn(v,S)$ of $v$ with respect to $S$ is defined to be $\epn(v,S)=\{u\in V(G)\setminus S: \, N(u)\cap S=\{v\}\}$.
The \emph{perfect neighbourhood of a set}  $S\subseteq V(G)$ is defined to be $N_p(S)=\{v\in V(G)\setminus S:\, |N(v)\cap S|=1\}$. We define the \emph{perfect differential of a set} $S\subseteq V(G)$ as $\partial_p(S)=|N_p(S)|-|S|$. The \emph{perfect differential of a graph}, introduced by Cabrera Mart\'{i}nez  and Rodr\'{i}guez-Vel\'{a}zquez in  \cite{PerfectDifferential},
is defined  as
$$\partial_p(G)=\max \{\partial_p(S):\, S\subseteq V(G)\}.$$

Now, $S\subseteq V(G)$ is a \emph{perfect dominating set} of $G$ if
every vertex  in $V(G)\setminus S$ is adjacent to exactly one vertex in $S$.
Let $\mathcal{D}^p(G)$ be the set of perfect dominating sets of $G$.
The \emph{perfect domination number} of $G$ is defined to be,
$$\gamma^p(G)=\min\{|S|:\, S\in \mathcal{D}^p(G)\}.$$
 Notice that  $\mathcal{D}^p(G)\subseteq \mathcal{D}(G)$, which implies that $\gamma(G)\le \gamma^{p}(G)$.

\begin{theorem}{\rm \cite{PerfectDifferential}}\label{Prop-RElat-Differentials}
Given a nontrivial graph $G$, the following inequality chain holds.
$$\n(G)-2\gamma^p(G) \le \partial_p(G)\le\frac{\n(G)(\Delta(G)-1)}{\Delta (G)+1}.$$
Furthermore, $\partial_p(G)=\frac{\n(G)(\Delta(G)-1)}{\Delta(G)+1}$   if and only if $\gamma^p(G)=\frac{\n(G)}{\Delta(G)+1}$.
\end{theorem}

The following theorems show some relationships between the $2$-packing differential and the perfect differential.

\begin{theorem}\label{ThPerfectDiffPackingDiff}
For any nontrivial graph $G$, the following statements hold.
\begin{enumerate}[{\rm (i)}]
\item $\partial_{2p}(G)\le \partial_p(G).$
\item $\partial_{2p}(G)\le\frac{\n(G)(\Delta(G)-1)}{\Delta (G)+1}.$
\item $\partial_{2p}(G)=\frac{\n(G)(\Delta(G)-1)}{\Delta (G)+1}$ if and only if $\gamma^p(G)=\frac{\n(G)}{\Delta(G)+1}$.
\end{enumerate}
\end{theorem}

\begin{proof}
Let $S$ be a $\partial_{2p}(G)$-set. Since $S$ is a $2$-packing, $N(S)=N_p(S)$. Hence,
$\partial_{2p}(G)=\partial(S)=|N(S)|-|S|=|N_p(S)|-|S|=\partial_p(S)\le \partial_p(G)$. Therefore, (i) follows.

Notice that (ii) is a direct consequence of (i) and Theorem \ref{Prop-RElat-Differentials}. We proceed to prove (iii).
It is known that $\gamma^p(G)\ge \gamma(G)\ge \frac{\n(G)}{\Delta(G)+1}$ and  $\gamma(G)=\frac{\n(G)}{\Delta(G)+1}$ if and only if   there exists a $\gamma(G)$-set $S$ which is a $\rho(G)$-set and every vertex in $S$ has degree $\Delta(G)$, for instance, see \cite{Haynes1998}. Hence, by Theorem \ref{th-packings}, if $\gamma^p(G)=\frac{\n(G)}{\Delta(G)+1}$, then
$$\partial_{2p}(G)=\displaystyle\max_{S\in \wp(G)}\left\{\sum_{v\in S}(deg(v)-1)\right\}=\rho(G)(\Delta(G)-1)=\frac{\n(G)(\Delta(G)-1)}{\Delta(G)+1}.$$

Conversely, if $\partial_{2p}(G)=\frac{\n(G)(\Delta(G)-1)}{\Delta(G)+1}$, then by (i) and Theorem \ref{Prop-RElat-Differentials} we have that $\partial_p(G)=\frac{\n(G)(\Delta(G)-1)}{\Delta (G)+1}$.  Therefore,  Theorem \ref{Prop-RElat-Differentials} leads to $\gamma^p(G)=\frac{\n(G)}{\Delta(G)+1}$.
\end{proof}

The difference $\partial_p(G)-\partial_{2p}(G)$ can be as large as desired. Consider the graph obtained by attaching to each vertex of a $5$-cycle a set of $k$ independent vertices. The resulting graph $G$ has order $n(G)=5k+5$ and satisfies $\partial_p(G)=5k-5$ (the set of vertices of the $5$-cycle is a $\partial_p(G)$-set) and $\partial_{2p}(G)=k+1$ (each single vertex in the $5$-cycle is a $\partial_{2p}(G)$-set). Moreover, the set of vertices of the $5$-cycle is a minimum perfect dominating set, so $\gamma^p(G)=5$ and $\partial_p(G)=n(G)-2\gamma^p(G)$ takes the smallest possible value.

Regarding the case in which both parameters agree, we have the following result.

\begin{theorem}
For any graph   $G$,  at least one of the following statements hold.
\begin{enumerate}
\item[{\rm (i)}] $\partial_{2p}(G)=\partial_p(G)$.
\item[{\rm (ii)}] $\partial_{2p}(G)\geq \frac{1}{2}(\Delta(G)\partial_p(G)+2 -(\n(G)-2)(\Delta(G)-2))$.
\end{enumerate}
\end{theorem}

\begin{proof}
By Theorem \ref{ThPerfectDiffPackingDiff} we have that $\partial_{2p}(G)\leq \partial_p(G)$. From now on, we assume that $\partial_{2p}(G)<\partial_p(G)$.
Let $S$ be a $\partial_p(G)$-set. Since $\partial_p(G)=|N_p(S)|-|S|$ and $|N_p(S)|+|S|\leq \n(G)$, it follows that $|S|\leq (\n(G)-\partial_p(G))/2$. Let $S'\subseteq S$ be maximal $2$-packing and $S''=S\setminus S'$. Observe that $S$ is not a $2$-packing, as $\partial_{2p}(G)<\partial_p(G)=\partial(S)$.  Hence, $N[S'']\cap N(S')\ne \varnothing$, which implies that
$$|N_p(S)|= |N(S')\setminus N[S'']|+\sum_{x\in S''}|epn(x,S)|\le (|N(S')|-1)+|S''|(\Delta(G)-1).$$
Thus, \begin{align*}
\partial_{p}(G)&=  |N_p(S)|-|S|\\
        &\leq (|N(S')|-1)+|S''|(\Delta(G)-1)-|S|\\
        &= \partial(S')+|S'|-1+|S''|(\Delta(G)-1)-|S|\\
        &= \partial(S')-1+|S''|(\Delta(G)-2)\\
        &\leq \partial_{2p}(G)-1+(|S|-1)(\Delta(G)-2)\\
        &\leq \partial_{2p}(G)-1+\left(\frac{\n(G)-\partial_p(G)}{2}-1\right)(\Delta(G)-2).
\end{align*}
Therefore, the result follows.
\end{proof}

The bound above is tight. For instance,  for any integer $t\geq 2$,  the double star $S_{t,t}$  satisfies that $\partial_p(S_{t,t})=2t-2$, $ \partial_{2p}(S_{t,t})=t$, $\n(S_{t,t})=2t+2$ and $\Delta(S_{t,t})=t+1$.  Hence, the bound is achieved by any double star $ S_{t,t}$ with $t\ge 2$.

\section{Unique response Roman domination versus $2$-pa\-cking differential }\label{SectionGallai}

In this section we establish a Gallai-type theorem which states the relationship between the $2$-packing differential and the unique response Roman domination number.

Cockayne, Hedetniemi and Hedetniemi \cite{Cockayne2004} defined a {\it Roman dominating  function}  on a graph $G$ to be a function $f: V(G)\longrightarrow \{0,1,2\}$ satisfying the condition that every vertex $u$ for which $f(u)=0$ is adjacent to at least one vertex $v$ for which $f(v)=2$. For  $X\subseteq V(G)$ we define the weight of $X$  as  $f(X)= \sum_{v\in X}f(v)$.  The \textit{weight} of $f$ is defined to be
 $$\omega(f)=f(V(G)).$$

The unique response  version of Roman domination was introduced  by Rubalcaba and Slater  in \cite{MR2370121} and studied further in \cite{MR2794312,MR3013246,MR3790857}.
 A function $f: V(G)\longrightarrow \{0,1,2\}$ with the sets $V_0, V_1, V_2$, where $V_i = \{v \in  V(G): \,  f (v) = i\}$
for $i \in\{0, 1, 2\}$, is a \emph{unique response Roman dominating function} if $x \in V_0$ implies $|N(x) \cap V_2| = 1$
and $x \in  V_1 \cup V_2$ implies that $N(x) \cap V_2=\varnothing$.  The \emph{unique response Roman domination number} of
$G$, denoted by $\mu_{_R}(G)$, is the minimum weight among all unique response Roman dominating
functions on $G$.

\begin{theorem}[Gallai-type theorem]\label{Th-Gallai-Roman-unique-response}
For any graph $G$,
$$\mu_{_R}(G)+\partial_{2p}(G)=\n(G).$$
\end{theorem}

\begin{proof}It was shown in \cite{MR2794312} that $\mu_{_R}(G)=\displaystyle\min_{S\subseteq \wp(G)}\{2|S|+|V(G)\setminus (N(S)\cup S)|\}$. Hence, the result is deduced as follows.
\begin{align*}
\mu_{_R}(G)&=\min_{S\subseteq \wp(G)}\{2|S|+|V(G)\setminus (N(S)\cup S)|\}\\
&= \min_{S\subseteq \wp(G)}\{2|S|+\n(G)-(|N(S)|+|S|)|\}\\
&=  \min_{S\subseteq \wp(G)}\{\n(G)-\partial(S)\}\\
&= \n(G)- \max_{S\subseteq \wp(G)}\{\partial  (S)\}\\
&= \n(G)-\partial_{2p}(G).
\end{align*}
\end{proof}

Theorem \ref{Th-Gallai-Roman-unique-response} allows us to derive results on the unique response Roman domination number from  results on the $2$-packing differential  and vice versa. For instance, from Corollary \ref{2-packing differential} and Theorem \ref{Th-Gallai-Roman-unique-response} we deduce the following result.

\begin{theorem}
The problem of finding the unique response Roman domination number of a graph is NP-hard, even for regular bipartite  graphs.
\end{theorem}

Theorem \ref{Th-Gallai-Roman-unique-response} suggests the challenge of obtaining new results on the unique response Roman domination number from the approach of $2$-packing differentials. As an example, the following table summarizes some of those results obtained here. The first column describes the result that  combined with Theorem \ref{Th-Gallai-Roman-unique-response} leads to the result on the second column.

\vspace{0,3cm}
\begin{center}
\begin{tabular}{|l|l|}
\hline
\rule[-1ex]{0pt}{4ex} \cellcolor{gray!10} From & \cellcolor{gray!10} Result for graphs with no isolated vertex\\
\hline
Theorem \ref{th-packings} (iii) & $\mu_{_R}(G)\ge \n(G)- \rho(G)(\Delta(G)-1)$.
\\&
 The equality holds if and only if there exists a $\rho(G)$-set $S$ such that \\& $\deg(v)=\Delta(G)$ for every $v\in S$.
\\
\hline
\rule[-1ex]{0pt}{4ex} Theorem \ref{Bound size and order} & $   \mu_{_R}(G)\ge \n(G)(\delta(G)+1)-\rho(G)(\delta(G)-1)-2\m(G)$.
\\&
 The equality holds if and only if there exists a $\rho(G)$-set $S$ such that
\\&
  $\deg(v)=\delta(G)$ for every $v\in V(G)\setminus S$.
\\
\hline
\rule[-1ex]{0pt}{4ex} Theorem \ref{Lower bound packing=dominating}  &
If $G$ is an efficient closed domination graph, then
$\mu_{_R}(G)\le 2\rho(G).$
\\& In particular, if there exists a $\rho(G)$-set  $S$ such that $\deg(v)=\delta(G)$ \\& for every $v\in V(G)\setminus S$,
then
$\mu_{_R}(G)= 2\rho(G).$
\\
\hline
\rule[-1ex]{0pt}{4ex} Theorem \ref{ThDominationIndep2} &
$ \mu_{_R}(G)\geq \frac{\n(G)+(\Delta(G)-1)i(G)}{\Delta(G)}.$
  \\
\hline
\rule[-1ex]{0pt}{4ex} Theorem \ref{ThDominationIndep2-particular-case} &
$ \mu_{_R}(G)\geq i(G)+1.$\\&
The equality holds if and only if $\Delta(G)=\n(G)-i(G)$.
  \\
\hline
\rule[-1ex]{0pt}{4ex} Theorem \ref{BounGammaTotal} &
$ \mu_{_R}(G)\ge \gamma_t(G).$\\&
The equality holds  if and only if $G$ is an efficient closed domination \\& graph   and  $\gamma_t(G)=2\gamma(G)$.
\\
\hline
\rule[-1ex]{0pt}{4ex} Theorem \ref{Second bound size and order}  &
$   \mu_{_R}(G)\ge  \frac{\n(G)(\delta(G)+1)-2\m(G)-\rho(G)(\delta(G)-1)+\gamma_t(G)\delta(G)}{\delta(G)+1}.$
\\&
 The equality holds if and only if there exists  a set  \\&  $S\in \mathcal{D}(G)\cap \wp(G)$ such that
  $\deg(v)=\delta(G)$ for every $v\in V(G)\setminus S$.\\
\hline
\rule[-1ex]{0pt}{4ex}
Theorem \ref{ThPerfectDiffPackingDiff} (i) &  $\mu_{_R}(G)\ge  \n(G)-\partial_p(G)$.  \\
\hline
Theorem \ref{ThPerfectDiffPackingDiff} (ii) &  $\mu_{_R}(G)=\frac{2\n(G)}{\Delta (G)+1}$ if and only if $\gamma^p(G)=\frac{\n(G)}{\Delta(G)+1}$.  \\
\hline
\end{tabular}
\end{center}

\section{The case of lexicographic product graphs}\label{Sectionlexicographic}

In this final section we present the behaviour of the $2$-packing differential under the lexicographic product operation. We first recall the following definition. Let $G$ and $H$ be two graphs.  The \emph{lexicographic product} of $G$ and $H$ is the graph $G \circ H$ whose vertex set is  $V(G \circ H)=  V(G)  \times V(H )$ and $(u,v)(x,y) \in E(G \circ H)$ if and only if $ux \in E(G)$ or $u=x$ and $vy \in E(H)$.
Notice that  for any $u\in V(G)$  the subgraph of $G\circ H$ induced by $\{u\}\times V(H)$ is isomorphic to $H$. For simplicity, we will denote this subgraph by $H_u$.

For a basic introduction to the lexicographic product of two graphs we suggest the  books  \cite{Hammack2011,Imrich2000}.
One of the main problems in the study of $G\circ H$ consists of finding exact values or tight
bounds of specific parameters of these graphs and express
them in terms of known invariants of $G$ and $H$. In particular,   we cite the following works on domination theory of lexicographic product graphs:   (total) domination   \cite{MR3363260,Nowakowski1996,Zhang2011},   Roman domination  \cite{SUmenjak:2012:RDL:2263360.2264103},   weak Roman domination   \cite{Valveny2017},   rainbow domination   \cite{MR3057019},   super domination    \cite{Dettlaff-LemanskaRodrZuazua2017}, doubly connected domination   \cite{MR3200151}, secure domination \cite{SecureLexicographicDMGT}, double domination \cite{DD-lexicographic}  and total Roman domination \cite{Dorota2019,TRDF-Lexicographic-2020}.

The following claim, which  states the distance formula in the lexicographic product of two graphs,  is one of our main tools.

\begin{remark}\label{claimLexi}{\rm \cite{Hammack2011}}
For any     connected graph $G$ of order $\n(G)\ge 2$ and any   graph $H$, the follo\-wing statements hold.
\begin{enumerate}[{\rm (i)}]
\item  $d_{G\circ H}((g,h),(g',h')) = d_{G}(g,g')$ for $g\ne g'$.
\item  $d_{G\circ H}((g,h),(g,h')) = \min\{2,d_H(h,h')\}$.
\end{enumerate}
\end{remark}

Given a set $W\subseteq V(G)\times V(H)$, the projection of $W$ on $V(G)$ will be denoted by $P_G(W)$. The following corollary is a direct consequence of the previous remark.

\begin{corollary}\label{CorollaryPackingLexic}
Let $G$ be a connected graph of order $\n(G)\ge 2$ and let $H$ be  a graph. A set  $W\subseteq V(G)\times V(H)$ is a $2$-packing of $G\circ H$ if and only if $P_G(W)$ is a $2$-packing  of $G$ and $|W|=|P_G(W)|$.
\end{corollary}

In the following result we show general lower and upper bounds for the $2$-packing differential of a lexicographic product graph, in terms of some parameters of both factors.

We define the following parameter.
 $$\rho'(G)=\max\{|S|: \,  S \text{ is a } \partial_{2p}(G)\text{-set} \}.$$

\begin{theorem}\label{ThLexicographic}
For any connected graph $G$ of order $\n(G)\ge 2$ and any   graph $H$,
$$ \n(H)\partial_{2p}(G)+\rho'(G)(\n(H)+\Delta(H)-1)\le \partial_{2p}(G\circ H)\le \n(H)\partial_{2p}(G)+\rho(G)(\n(H)+\Delta(H)-1).$$
Furthermore,
$$ \n(H)\mu_{_R}(G)-\rho(G)(\n(H)+\Delta(H)-1)\le \mu_{_R}(G\circ H)\le \n(H)\mu_{_R}(G)-\rho'(G)(\n(H)+\Delta(H)-1).$$
\end{theorem}

\begin{proof}
Let $W\subseteq V(G)\times V(H)$ be a $\partial_{2p}(G\circ H)$-set and $U=P_G(W)$.
By Corollary \ref{CorollaryPackingLexic}, we learned that $U$ is a $2$-packing of $G$. 
Hence,
\begin{align*}
\partial_{2p}(G\circ H)&= |N(W)|-|W|\\
  &=\sum_{(g,h)\in W}|N(g)|\n(H)+\sum_{(g,h)\in W}\deg(h)-|W|\\
   &\le |N(U)|\n(H)+|U|\Delta(H)-|U|\\
   &= \partial(U)\n(H)+|U|(\n(H)+\Delta(H)-1)\\
   &\le  \partial_{2p}(G)\n(H)+\rho(G)(\n(H)+\Delta(H)-1)).
\end{align*}
Now, for any $\partial_{2p}(G)$-set $S$ with $|S|=\rho'(G)$ and any $v\in V(H)$ with $\deg(v)=\Delta(H)$, the set $S\times \{v\}$ is a $2$-packing of $G\circ H$, and so
\begin{align*}
\partial_{2p}(G\circ H)&\ge  \partial(S\times \{v\})\\
  &=\sum_{u\in S}(|N(u)|\n(H)+\Delta(H))-\rho'(G)\\
  &= |N(S)|\n(H)+\Delta(H)\rho'(G)-\rho'(G)\\
  &= \partial(S)\n(H)+\rho'(G)(\n(H)+\Delta(H)-1)\\
   &=\partial_{2p}(G)\n(H)+\rho'(G)(\n(H)+\Delta(H)-1).
\end{align*}
By Theorem \ref{Th-Gallai-Roman-unique-response} we complete the proof.
\end{proof}

If $G$ is a regular graph, then  $\rho(G)=\rho'(G)$. Hence, from Theorem  \ref{ThLexicographic} we derive the following result.

\begin{corollary}
The following statements hold for any connected regular graph $G$ of order $\n(G)\ge 2$ and any   graph $H$.

\begin{itemize}
\item $\partial_{2p}(G\circ H)= \n(H)\partial_{2p}(G)+\rho(G)(\n(H)+\Delta(H)-1).$

\item $\mu_{_R}(G\circ H)= \n(H)\mu_{_R}(G)-\rho(G)(\n(H)+\Delta(H)-1).$
\end{itemize}

\end{corollary}

If $G$ is an efficient closed domination graph with
$\partial_{2p}(G)= \n(G)-2\rho(G)$, then  $\rho(G)=\rho'(G)$. Therefore, Theorem  \ref{ThLexicographic} leads to the following result.

\begin{corollary}
If $G$ is a connected  efficient closed domination graph with $\n(G)\ge 2$ and
$\partial_{2p}(G)= \n(G)-2\rho(G)$, then the following statements hold for any   graph $H$,
\begin{itemize}
\item $  \partial_{2p}(G\circ H)= \n(H)\partial_{2p}(G)+\rho(G)(\n(H)+\Delta(H)-1).$
\item $\mu_{_R}(G\circ H)= \n(H)\mu_{_R}(G)-\rho(G)(\n(H)+\Delta(H)-1).$
\end{itemize}
 \end{corollary}

We finish with the computation of the exact values of the $2$-packing differential and the unique response Roman domination number of the lexicographic product of a path and any other graph.

\begin{theorem}\label{PrH-Nonempty}
Let $r\geq 2$ be an integer. If $H$ is  a nonempty graph, then
$$\partial_{2p}(P_r\circ H)=\displaystyle\left\{ \begin{array}{ll}
                     \frac{r}{3}\left( 2\n(H)+\Delta(H)-1\right), &  r\equiv 0 \pmod 3,\\[6pt]
                    \frac{2}{3}(r-1)\n(H)+\frac{r+2}{3}(\Delta(H)-1), &  r\equiv 1 \pmod 3,\\[6pt]
                     \frac{2r-1}{3}  \n(H)+\frac{r+1}{3}(\Delta(H)-1), &   r\equiv 2 \pmod 3.
                                  \end{array}\right.
                                  $$
                                   Furthermore,
                                  $$\mu_{_R}(P_r\circ H)=\displaystyle\left\{ \begin{array}{ll}
                     \frac{r}{3}\left( \n(H)-\Delta(H)+1\right), &  r\equiv 0 \pmod 3,\\[6pt]
                    \frac{r+2}{3}(\n(H)-\Delta(H)+1), &  r\equiv 1 \pmod 3,\\[6pt]
                     \frac{r+1}{3}  (\n(H)-\Delta(H)+1), &   r\equiv 2 \pmod 3.
                                  \end{array}\right.
                                  $$
\end{theorem}

\begin{proof}
If $r\equiv 0,2 \pmod 3$, then $\rho'(P_r)=\rho(P_r)$. Hence, the result follows from Theorem \ref{ThLexicographic}. Now, assume  $r\equiv 1 \pmod 3$,
and let $S$ be a $\rho(P_r)$-set. In this case,
 $P_r$ is an efficient closed  domination graph, and so $|N(S)|=r-|S|=\frac{2(r-1)}{3}$. Thus, for every vertex $v\in V(H)$ with $\deg(v)=\Delta(H)$,
 \begin{align*}
\partial_{2p}(P_r\circ H)&\ge  \partial(S\times \{v\})\\
  &=\sum_{u\in S}(|N(u)|\n(H)+\Delta(H)-1)\\
  &= |N(S)|\n(H)+\rho(P_r)(\Delta(H)-1)\\
  &=  \frac{2}{3}(r-1)\n(H)+\frac{r+2}{3}(\Delta(H)-1).
\end{align*}
It remains to show that the equality holds.  To this end, let  $W$ be a $\partial_{2p}(P_r\circ H)$-set and $U$ the projection of $W$ on $V(P_r)$.
By Corollary \ref{CorollaryPackingLexic}, we learned that $U$ is a $2$-packing of $P_r$.
Hence, as in the proof of Theorem \ref{ThLexicographic} we deduce that $\partial_{2p}(P_r\circ H)\le  |N(U)|\n(H)+|U|(\Delta(H)-1).$
Thus, if $U$ is a $\rho(P_r)$-set, then $\partial_{2p}(P_r\circ H)=\frac{2}{3}(r-1)\n(H)+\frac{r+2}{3}(\Delta(H)-1)$, as required. Now, suppose to the contrary, that  $|U|\le \rho(P_r)-1=\frac{r-1}{3}$.
In such a case, $|N(U)|\le 2|U|\le \frac{2(r-1)}{3}$, and so
\begin{align*}
\frac{2}{3}(r-1)\n(H)+\frac{r+2}{3}(\Delta(H)-1)&\le \partial_{2p}(P_r\circ H)\\
 &\le |N(U)|\n(H)+|U|(\Delta(H)-1) \\
 &\le \frac{2}{3}(r-1)\n(H)+\frac{r-1}{3}(\Delta(H)-1),
\end{align*}
which is a contradiction for $\Delta(H)\ge 2$, and we have equalities for $\Delta(H)=1$.

By Theorem \ref{Th-Gallai-Roman-unique-response} we complete the proof.
\end{proof}

We can make the appropriate modifications in the proof of the previous theorem  to show that the following result follows.

\begin{theorem}\label{PrH-Empty}
Let $r\geq 2$ be an integer. If $H$ is  an empty graph, then
$$\partial_{2p}(P_r\circ H)=\displaystyle\left\{ \begin{array}{ll}
                     \frac{r}{3}\left( 2\n(H)-1\right), &  r\equiv 0 \pmod 3,\\[6pt]
                    \frac{r-1}{3}(2\n(H)-1), &  r\equiv 1 \pmod 3,\\[6pt]
                     \frac{2r-1}{3}  \n(H)-\frac{r+1}{3}, &   r\equiv 2 \pmod 3.
                                  \end{array}\right.
                                  $$
                                  Furthermore,
                                  $$\mu_{_R}(P_r\circ H)=\displaystyle\left\{ \begin{array}{ll}
                     \frac{r}{3}\left( \n(H)+1\right), &  r\equiv 0 \pmod 3,\\[6pt]
                    \frac{r+2}{3}\n(H)+\frac{r-1}{3}, &  r\equiv 1 \pmod 3,\\[6pt]
                     \frac{r+1}{3} ( \n(H)+1), &   r\equiv 2 \pmod 3.
                                  \end{array}\right.
                                  $$
\end{theorem}

\end{document}